\documentclass[a4paper,12pt]{amsart}

\usepackage{amsfonts}
\usepackage{amsmath}
\usepackage{amssymb}

\usepackage{mathrsfs}
\usepackage{hyperref}

\numberwithin{equation}{section}
\theoremstyle{plain}
\newtheorem{thm}{Theorem}[section]

\newtheorem{lemma}[thm]{Lemma}

\theoremstyle{definition}

\newtheorem{rem}[thm]{Remark}

\newcommand{\be}{\begin{equation}}
\newcommand{\ee}{\end{equation}}

\def\R{{\mathbb R}}

\def\C{{\mathbb C}}
\def\S3{{{\mathbb S}^3}}
\def\SU2{{{\rm SU}(2)}}
\def\Rn{{{\mathbb R}^n}}
\def\Tn{{{\mathbb T}^n}}

\def\Gh{{\widehat{G}}}
\def\Lap{{\mathcal L}}

\def\K{{\mathscr K}}
\def\L{{\mathscr L}}

\def\HS{{\mathtt{HS}}}

\def\p#1{{\left({#1}\right)}}

\def\jp#1{{\left\langle{#1}\right\rangle}}

\def\Dcal{{\mathcal D}}

\DeclareMathOperator{\Tr}{Tr}

\DeclareMathOperator{\rank}{rank}

\def\C{{\mathbb C}}

\def\Rn{{\mathbb R}^n}
\def\R2n{{\mathbb R}^{2n}}

\def\S{{\mathcal S}}

\def\Rn{{\mathbb R}^n}

\def\C{{\mathbb C}}

\def\R2{{\mathbb R}^2}
\def\R2n{{\mathbb R}^{2n}}

\def\S{{\mathcal S}}

\begin{document}

\title[Gohberg lemma and essential spectrum on compact Lie groups]
{Gohberg lemma, compactness, and essential spectrum of operators on compact Lie groups}

\author[Aparajita Dasgupta]{Aparajita Dasgupta}
\address{
  Aparajita Dasgupta:
  \endgraf
  Department of Mathematics
  \endgraf
  Imperial College London
  \endgraf
  180 Queen's Gate, London SW7 2AZ 
  \endgraf
  United Kingdom
  \endgraf
  {\it E-mail address} {\rm a.dasgupta@imperial.ac.uk}
  }

\author[Michael Ruzhansky]{Michael Ruzhansky}
\address{
  Michael Ruzhansky:
  \endgraf
  Department of Mathematics
  \endgraf
  Imperial College London
  \endgraf
  180 Queen's Gate, London SW7 2AZ 
  \endgraf
  United Kingdom
  \endgraf
  {\it E-mail address} {\rm m.ruzhansky@imperial.ac.uk}
  }

\thanks{The first author was supported by the Grace-Chisholm Young Fellowship of the 
London Mathematical Society. The second
 author was supported by the EPSRC
 Leadership Fellowship EP/G007233/1 and by EPSRC Grant EP/K039407/1.}
\date{\today}


\begin{abstract}
In this paper we prove a version of the Gohberg lemma on compact Lie groups
giving an estimate from below for the distance from a given operator to the set
of compact operators on compact Lie groups. As a consequence, we prove
several results on bounds for the essential spectrum and a criterion for
an operator to be compact. The conditions are
given in terms of the matrix-valued symbols of operators.
\end{abstract}

\maketitle

\section{Introduction}

In this paper we establish a version of the Gohberg lemma in the setting of
compact Lie groups and apply it to study the compactness of 
pseudo-differential operators and give bounds for their essential spectrum.
The original Gohberg lemma has been obtained by Gohberg 
\cite{Gohberg:lemma-1960} in the investigation of integral operators,
and its version on the unit circle $\mathbb T^{1}$ has been recently obtained
by \cite{Molahajloo-Wong:eFredholmness-2010}, with application to
the spectral properties of operators, see 
\cite{Molahajloo:torus-compact-2011,Pirhayati:ess-spectrum-2011}.
In this paper we establish the Gohberg lemma on general compact Lie
groups using the matrix quantization of operators
developed in \cite{ruzhansky+turunen-book,Ruzhansky+Turunen-IMRN}.
In particular, we give estimates for the distance from a given operator to
the set of compact operators as well as for the essential spectrum of
the operator in terms of some quantities associated to the matrix
symbols. The results contain the corresponding results obtained in
\cite{Molahajloo:torus-compact-2011,Pirhayati:ess-spectrum-2011}
on the unit circle. The matrix-valued symbols have been quite useful in
other studies of compactness of operators
in cases when 
conditions on the kernel are less effective, for example by providing
criteria for operators to belong to Schatten classes, see 
\cite{Delgado-Ruzhansky:Schatten},
and criteria for nuclearity in $L^{p}$-spaces, see
\cite{Delgado-Ruzhansky:Lp-nuclearity}.

In Section \ref{SEC:Prelim} we briefly recall the necessary notions of the
Fourier analysis on compact Lie groups and of the matrix quantization
of operators. In Section \ref{SEC:Results} we state our results.
In Section \ref{SEC:proof1} we prove the Gohberg lemma given in
Theorem \ref{THM:Gohberg}, and in
Section \ref{SEC:other-proofs} we prove an application of the 
Gohberg lemma given in Theorem \ref{THM:Thmpm}.

\section{Fourier analysis and matrix symbols on compact Lie groups}
\label{SEC:Prelim}

Let $G$ be a compact Lie group with the unit element $e$,
and let $\Gh$ be its unitary dual, consisting of the
equivalence classes $[\xi]$ of the continuous irreducible unitary representations
$\xi:G\to \C^{d_{\xi}\times d_{\xi}}$ of dimension $d_{\xi}$. For a function
$f\in C^{\infty}(G)$ we can define its Fourier coefficient at $\xi$ by
$$
\widehat{f}(\xi):=\int_{G} f(x) \xi(x)^{*}dx\in\C^{d_{\xi}\times d_{\xi}},
$$
where the integral is (always) taken with respect to the Haar measure on $G$.
The Fourier series becomes
$$
f(x)=\sum_{[\xi]\in\Gh} d_{\xi} \Tr\p{\xi(x)\widehat{f}(\xi)},
$$
with the Plancherel's identity taking the form
\begin{equation}\label{EQ:Plancherel}
\|f\|_{L^{2}(G)}=\p{\sum_{[\xi]\in\Gh} d_{\xi} \|\widehat{f}(\xi)\|_{\HS}^{2}}^{1/2}=:
\|\widehat{f}\|_{\ell^{2}(\Gh)},
\end{equation}
which we take as the definition of the norm on the Hilbert space
$\ell^{2}(\Gh)$, and where 
$ \|\widehat{f}(\xi)\|_{\HS}^{2}=\Tr(\widehat{f}(\xi)\widehat{f}(\xi)^{*})$ is the
Hilbert--Schmidt norm of the matrix $\widehat{f}(\xi)$.

Given an operator $T:C^{\infty}(G)\to C^{\infty}(G)$
(or even $T:C^{\infty}(G)\to \Dcal'(G)$), we define its matrix symbol by
$\sigma_{T}(x,\xi):=\xi(x)^{*} (T\xi)(x)\in \C^{d_{\xi}\times d_{\xi}}$, where
$T\xi$ means that we apply $T$ to the matrix components of $\xi(x)$.
In this case we can prove that 
\begin{equation}\label{EQ:T-op}
Tf(x)=\sum_{[\xi]\in\Gh} d_{\xi} \Tr\p{\xi(x)\sigma_{T}(x,\xi)\widehat{f}(\xi)}.
\end{equation}
The correspondence between operators and symbols is one-to-one, and
we will write $T_{\sigma}$ for the operator given by 
\eqref{EQ:T-op} corresponding to the symbol $\sigma(x,\xi)$.
The quantization \eqref{EQ:T-op} has been extensively studied in
\cite{ruzhansky+turunen-book,Ruzhansky+Turunen-IMRN}, to which we
refer for its properties and for the corresponding symbolic calculus.

We note that the matrix components of $\xi(x)$ are the eigenfunctions of the
Laplacian (Casimir element) $\Lap$ on $G$ corresponding to one eigenvalue which we
denote by $\lambda_{\xi}^{2}$, i.e. we have
$\Lap \xi(x)_{ij}=-\lambda_{\xi}^{2}\xi(x)_{ij}$ for all $1\leq i,j\leq d_{\xi}$.
We denote $\jp{\xi}:=(1+\lambda_{\xi}^2)^{1/2}.$

We now briefly describe the class $\Psi^{0}(G)$
of H\"ormander's pseudo-differential operators on $G$
in terms of the matrix symbols. Here, $\Psi^{0}(G)$ stands for the
usual class of operators that have symbols in H\"ormander's class
$S^{0}_{1,0}(\Rn)$ in every local coordinate system.

It was proved in \cite{Ruzhansky-Turunen-Wirth:arxiv} that $T\in \Psi^{0}(G)$ is equivalent 
to the condition that its matrix-valued symbol $\sigma$ satisfies
\begin{equation}\label{EQ:symbol-class}
\|\partial_{x}^{\beta}\Delta_{\xi}^{\alpha}\sigma(x,\xi)\|_{op}\leq C_{\alpha\beta}
\jp{\xi}^{-|\alpha|}
\end{equation}
for all $x\in G$ and $[\xi]\in\Gh$, and for all $\alpha,\beta$,
where $\|\cdot\|_{op}$ stands for the operator norm of the matrix
multiplication.
The difference operators
$\Delta_{\xi}^{\alpha}$ in \eqref{EQ:symbol-class} are defined as follows.
Let $q_1,\ldots,q_m\in C^{\infty}(G)$ be such that $q_{j}(e)=0$,
$\nabla q_{j}(e)\not=0$, for all $1\leq j\leq m$, 
the unit element $e$ is the only common zero of
the family $\{q_{j}\}_{j=1}^{m}$, and such that
$\rank\{\nabla q_{1}(e),\cdots,\nabla q_{m}(e)\}=\dim G$.
We call such a collection strongly admissible.
Then we set $\Delta_{q_{j}}\widehat f (\xi):=\widehat{q_j f}(\xi)$ 
and 
$\Delta^{\alpha}_{\xi}:=\Delta_{q_{1}}^{\alpha_{1}}\cdots \Delta_{q_{m}}^{\alpha_{m}}.$
We refer to \cite{ruzhansky+turunen-book} and especially to
\cite{Ruzhansky-Turunen-Wirth:arxiv} for the analysis of such difference operators.

It was shown in \cite{Ruzhansky-Turunen-Wirth:arxiv}
that the operator $T\in \Psi^{0}(G)$ is elliptic if and only if 
its matrix symbol $\sigma(x,\xi)$ is invertible for all but finitely many
$[\xi]\in\Gh$, and for all such $\xi$ we have
\begin{equation}\label{EQ:ellipticity}
\|\sigma(x,\xi)^{-1}\|_{op}\leq C
\end{equation}
for all $x\in G$.

\section{Gohberg lemma and applications}
\label{SEC:Results}
  
We define $\|\sigma(x,\xi)\sigma(x,\xi)^{*}\|_{\min}$ to be the smallest eigenvalue of 
the positive matrix $\sigma(x,\xi)\sigma(x,\xi)^{*},$ that is, if 
$\lambda_1(x,\xi), \lambda_2(x,\xi),\ldots,\lambda_{d_{\xi}}(x,\xi)\geq 0$ are the eigenvalues of 
$\sigma(x,\xi)\sigma(x,\xi)^{*}$
then we set
$$
\|\sigma(x,\xi)\sigma(x,\xi)^{*}\|_{\min}:= \min_{1\leq i\leq d_{\xi}}\lambda_i(x,\xi).
$$
We formulate a version of the Gohberg Lemma first for operators in the 
H\"ormander class $\Psi^0$ to relate it with the well-known theory and to 
be used in the application in Theorem \ref{THM:Thmpm}, but later, in
Remark \ref{REM:Gohberg}, we note that the result remains valid for
a much more general class of operators.

\begin{thm}[Gohberg Lemma]\label{THM:Gohberg}
Let $\sigma(x,\xi)$ be the matrix symbol of $T_{\sigma}\in \Psi^{0}(G)$. 
Then for all compact operators $K$ on $L^{2}(G)$, we have 
\begin{equation}\label{EQ:Gohberg-est}
\|T_{\sigma}-K\|_{\mathscr L(L^{2}(G))}\geq d_{\min},
\end{equation}
where 
$$
d_{\min}:=\limsup_{\jp{\xi}\rightarrow\infty}
\left\{\sup_{x\in G}\frac{\|\sigma(x,\xi)\sigma(x,\xi)^{*}\|_{\min}}
{\|\sigma(x,\xi)\|_{op}}\right\}.
$$
\end{thm}
We note that $d_{\min}$ is well-defined. Indeed, we have
$$
\|\sigma(x,\xi)\sigma(x,\xi)^{*}\|_{\min}\leq 
\|\sigma(x,\xi)\sigma(x,\xi)^{*}\|_{op}\leq
\|\sigma(x,\xi)\|_{op}^{2},
$$
implying that 
\begin{equation}\label{EQ:dmin-est}
d_{\min}\leq \limsup_{\jp{\xi}\to\infty}\{\sup_{x\in G}\|\sigma(x,\xi)\|_{op}\}<\infty
\end{equation}
in view of \eqref{EQ:symbol-class} with $\alpha=\beta=0$.

We note again that the condition $T_{\sigma}\in \Psi^{0}(G)$ in
Theorem \ref{THM:Gohberg} can be substantially relaxed,
see Remark \ref{REM:Gohberg}.

\medskip
To formulate an application of the Gohberg lemma, let us first 
introduce some notation.
Let $A:X\rightarrow X$ be a closed linear operator with dense domain $D(A)$ in the complex Banach space $X.$ Then the spectrum $\Sigma(A)$ of $A$  is denoted by
$
\Sigma(A):={\mathbb{C}}\backslash\Phi(A),
$ 
where $\Phi(A)$ is the resolvent set of $A$ 
given by 
$$
\Phi(A)=\{\lambda\in\mathbb{C}: A-\lambda I \textrm{ is bijective}\}.
$$
The essential spectrum $\Sigma_{ess}(A)$ of $A$ is 
defined by $\Sigma_{ess}(A):=\mathbb{C}\backslash\Phi_{ess}(A),$ where 
$$
\Phi_{ess}(A)=\left\{\lambda\in\mathbb{C}: A-\lambda I
 \textrm{ is Fredholm and } i(A-\lambda I)=
0\right\}.
$$

\begin{thm} \label{PROP:1}\label{THM:Thmpm}
Let $\sigma$ be the matrix symbol of a pseudo-differential operator $T_{\sigma}\in\Psi^{0}(G).$
Let 
$$
d_{\max}:=\limsup_{\jp{\xi}\rightarrow \infty}\{\sup_{x\in G}\|\sigma(x,\xi)\|_{op}\}.
$$
Then for $T_{\sigma}$ on $L^{2}(G)$ we have 
\begin{equation}\label{EQ:ess}
\Sigma_{ess}(T_{\sigma})\subseteq \left\{\lambda\in{\mathbb {C}}: |\lambda|\leq d_{\max} \right\}.
\end{equation} 
Moreover, if $d_{\max}=0$, then $T_{\sigma}$ is a compact operator on $L^{2}(G)$.
\end{thm}
We observe that it follows from Theorem \ref{THM:Gohberg} that
if $d_{\min}\not= 0$, then $T_{\sigma}$ is not compact, since otherwise
we could take $K=T_{\sigma}.$ 
In other words, if $T_{\sigma}$ is compact, then $d_{\min}=0$.
From this point of view, the converse to this is
given by the last statement of Theorem \ref{THM:Thmpm}.

We note that we always have $d_{\min}\leq d_{\max}$ in view of 
\eqref{EQ:dmin-est}. If $G=\Tn$ is the torus, we have 
$d_{\min}= d_{\max}$.


\section{Proof of Theorem \ref{THM:Gohberg}}
\label{SEC:proof1}

This section is devoted to the proof of Theorem \ref{THM:Gohberg}.

\medskip
First we observe that by \eqref{EQ:dmin-est},
$d_{\min}$ is well-defined, and hence
for every $[\xi]\in \Gh$ there exists $x_{\xi}\in G$ such that 
$$
\frac{\|\sigma(x_{\xi_n},\xi_n)\sigma(x_{\xi_n},\xi_n)^{*}\|_{\min}}
{\|\sigma(x_{\xi_n},\xi_n)\|_{op}}=
\sup_{x\in G}\frac{\|\sigma(x,\xi_n)\sigma(x,\xi_n)^{\ast}\|_{\min}}{\|\sigma(x,\xi_n)\|_{op}}.
$$
From the definition of $d_{\min}$ there exists a sequence 
$(x_{\xi_n}, \xi_{n})$ such that 
$\jp{\xi_n}\rightarrow \infty$ and $
\frac{\|\sigma(x_{\xi_n},\xi_n)\sigma(x_{\xi_n},\xi_n)^{\ast}\|_{\min}}
{\|\sigma(x_{\xi_n},\xi_n)\|_{op}}\rightarrow d_{\min}.$ 
Let $u\in L^{2}(G)$ be sufficiently smooth.
We define $u_{\xi_n}(x)\in {\mathbb{C}}^{d_{\xi}\times d_{\xi}}$ by
$$
u_{\xi_n}(x):=d_{\xi_{n}}^{-1/2}\xi_{n}(x)u(x\cdot x_{\xi_n}^{-1}).
$$ 
For a matrix-valued function $w=w(x)\in {\mathbb{C}}^{d_{\xi}\times d_{\xi}}$, we define
its $L^{2}$-matrix norm by
$$
 \|w\|_{L^{2}(G)}:=\p{\int_{G}\|w(x)\|_{\HS}^{2}dx}^{1/2}.
$$
Then  we have 
\begin{eqnarray}
\|u_{\xi_n}\|^{2}_{L^{2}(G)}&=&\int_{G}\|u_{\xi_n}(x)\|_{\HS}^{2}dx\nonumber\\
&=&
\int_{G}d_{\xi_{n}}^{-1}\|\xi_n(x)u(x\cdot x_{\xi_n}^{-1})\|_{\HS}^{2}dx\nonumber\\
&=& d_{\xi_{n}}^{-1}\|\xi_n\|^2_{\HS}\|u\|^{2}_{L^{2}(G)}\nonumber\\
&=& 
\|u\|^{2}_{L^{2}(G)}.
\end{eqnarray} 
Therefore, we have the equality 
\begin{equation}\label{EQ:L2s}
\|u_{\xi_n}\|_{L^{2}(G)}=\|u\|_{L^{2}(G)}=\|\widehat{u}\|_{\ell^2(\Gh)}.
\end{equation}
Let now $\phi\in C^{\infty}(G)$. Then, with $y=x\cdot x_{\xi_{n}}^{{-1}},$ we have 
\begin{eqnarray}
\int_{G}u_{\xi_n}(x)\phi(x)dx 
&=& 
d^{-1/2}_{\xi_n}\int_{G}{\xi_{n}(y)u(y)\phi(y \cdot x_{\xi_n})\xi_{n}(x_{\xi_n})dy}\nonumber\\
&=& 
d_{\xi_n}^{-1/2}\widehat{u\phi(\cdot x_{\xi_n})}(\xi_n^{\ast})\xi_{n}(x_{\xi_n}),
\end{eqnarray}
where $\xi_{n}^{*}(x)=\xi_{n}(x)^{*}.$
Since
$$
\|u\phi(\cdot x_{\xi_n})\|_{L^2}\leq C \|u\|_{L^2}
$$
with a constant $C$ independent of $x_{\xi_{n}}$, we have
$\widehat{u\phi(\cdot x_{\xi_n})}\in \ell^{2}(\Gh)$ uniformly in $x_{\xi_n}$.
Hence, from \eqref{EQ:L2s} and \eqref{EQ:Plancherel}, it follows that
$d_{\xi_n}\|\widehat{u\phi(\cdot x_{\xi_n})}(\xi_n^{\ast})\|^{2}_{\HS}\rightarrow 0$ as 
$\jp{\xi_n^{\ast}}\rightarrow \infty.$
This implies
\begin{eqnarray}
\|\int_G u_{\xi_n}(x)\phi(x)dx\|_{\HS}=
\|d^{-1/2}_{\xi_n}\widehat{u\phi(\cdot x_{\xi_n})}(\xi^{\ast}_n)\xi_{n}(x_{\xi_n})\|_{\HS}
&\leq 
\|\widehat{u\phi(\cdot x_{\xi_n})}(\xi^{\ast}_n)\|_{\HS},\nonumber
\end{eqnarray}
so that
$$
u_{\xi_n}\rightarrow 0 \textrm{ as } \jp{\xi_n}\rightarrow \infty
$$ 
weakly. 
Hence for a compact  operator $K$ we have  
$$
\|Ku_{\xi_n}\|_{L^{2}(G)}\rightarrow 0
$$ 
as $\jp{\xi_n}\rightarrow \infty.$ 
Then for any $\epsilon>0$ and sufficiently large $n$ we have  by compactness
\begin{eqnarray}
\|Ku_{\xi_n}\|_{L^{2}(G)}&\leq& \epsilon \|u_{\xi_n}\|_{L^{2}(G)}\nonumber\\
&=& \epsilon \|u\|_{L^{2}(G)},
\end{eqnarray}
where $u$ is fixed and
$\|Ku_{\xi_n}\|_{L^{2}(G)}=\left(\int_{G}\|Ku_{\xi_n}(x)\|_{\HS}^2dx\right)^{1/2}.$
We now define 
$$
T_{\sigma}u_{\xi_n}:=\left(T_{\sigma}(u_{\xi_n})_{ij}\right)_
{1\leq i,j\leq d_{\xi_n}}\in {\mathbb{C}}^{d_{\xi_{n}}\times d_{\xi_n}}
$$
by $T_{\sigma}$ acting on the components of the matrix-valued function
$u_{\xi_n}.$

\begin{lemma}\label{LEM:Tsigma-u}
We have
$\|u_{\xi_n}\sigma(\cdot,\xi_n)-T_{\sigma}u_{\xi_n}\|_{L^{2}(G)}\rightarrow 0$ 
as $\jp{\xi_n}\rightarrow \infty.$
\end{lemma}
We postpone the proof of Lemma \ref{LEM:Tsigma-u} and continue with the
proof of Theorem \ref{THM:Gohberg}.

\medskip
Let us fix $u\in C^{\infty}(G)$ such that $u\not=0.$ 
Then for any $\epsilon> 0$ there exists $N(u)$ such that for any $n\geq N(u)$ 
we have 
\begin{equation}\label{EQ:pf-aux1}
\|u_{\xi_n}\sigma(\cdot,\xi_n)\|_{L^2(G)}-\|T_{\sigma}u_{\xi_n}\|_{L^{2}(G)}
\leq \epsilon\|u\|_{L^2(G)}
\end{equation}
for sufficiently large $\jp{\xi_n}.$ 
Now since $\sigma$ satisfies \eqref{EQ:symbol-class} with $\alpha=0$, its
$x$-derivatives are uniformly bounded, and hence for $\epsilon>0$ 
there exists an open neighbourhood $V$ of the unit $e$ of the group 
such that for all $x\cdot x_{\xi_{n}}^{-1}\in V\subseteq G$ we have
$$
\|\sigma(x,\xi_n)-\sigma(x_{\xi_n},\xi_n)\|_{op}<\epsilon.
$$
Let now $u\in C^{\infty}(G)$ be such that $u(x)=0$ for all $x\notin V.$ 
Then $u_{\xi_n}(x)=0$ for all $x\not\in x_{\xi_n}V$, i.e. for 
$x\cdot x_{\xi_n}^{-1}\notin V.$ 
Then
\begin{eqnarray}\label{EQ:aux1}
&&  
\frac{\|u_{\xi_n}\sigma(x_{\xi_n},\xi_n)\sigma(x_{\xi_n},\xi_n)^{\ast}\|_{L^{2}(G)}}
{\|\sigma(x_{\xi_n},\xi_n)\|_{op}}-\|u_{\xi_n}\sigma(\cdot,\xi_{n})\|_{L^{2}(G)}
\nonumber
\\
&\leq& \|u_{\xi_n}\sigma(x_{\xi_n},\xi_n)\|_{L^{2}(G)}-\|u_{\xi_n}\sigma(\cdot,\xi_n)
\|_{L^2(G)}\nonumber\\
&\leq&\|u_{\xi_n}\sigma(x_{\xi_n},\xi_n)-
u_{\xi_n}\sigma(\cdot,\xi_n)\|_{L^2(G)}\nonumber\\
&\leq& \left(\int_{x_{\xi_n} V}\|\sigma(x,\xi_n)-\sigma(x_{\xi_n},\xi_n)\|_{op}^2
\|u_{\xi_n}(x)\|_{\HS}^{2}dx\right)^{1/2}\nonumber\\
&\leq& \epsilon\left(\int_{x_{\xi_n}V}\|u_{\xi_n}(x)\|_{\HS}^2\right)^{1/2}
=\epsilon \|u_{\xi_n}\|_{L^{2}(G)}
=\epsilon \|u\|_{L^2(G)},
\end{eqnarray}
the last inequality following from \eqref{EQ:L2s}.
Therefore,
\begin{eqnarray}\label{EQ:aux4}
\|u\|_{L^2(G)}\|T_{\sigma}-K\|_{\L(L^{2}(G))}
&=& \|u_{\xi_n}\|_{L^2(G)}\|T_{\sigma}-K\|_{\L(L^{2}(G))}\nonumber\\
&\geq&\|(T_{\sigma}-K)u_{\xi_n}\|_{L^{2}(G)}\nonumber\\
&\geq& \|T_{\sigma}u_{\xi_n}\|_{L^{2}(G)}-\|Ku_{\xi_n}\|_{L^{2}(G)}\nonumber\\
&\geq& \|u_{\xi_n}\sigma{(\cdot,\xi_n)}\|_{L^{2}(G)}-2\epsilon\|u\|_{L^2(G)}\nonumber\\
&\geq& \frac{\|u_{\xi_n} \sigma(x_{\xi_n},\xi_n)\sigma(x_{\xi_n},\xi_n)^{\ast}\|_{L^2(G)}}
{\|\sigma(x_{\xi_n},\xi_n)\|_{op}}-3\epsilon \|u\|_{L^2(G)},
\end{eqnarray}
using \eqref{EQ:pf-aux1} and \eqref{EQ:aux1}
for the last inequalities.
Since
$\sigma(x_{\xi_n},\xi_n)\sigma(x_{\xi_n},\xi_n)^{\ast}$ is normal 
there exist a unitary matrix $U$ such that
$$\sigma(x_{\xi_n},\xi_n)\sigma(x_{\xi_n},\xi_n)^{\ast}=U\Lambda U^{\ast},$$ 
where
$$
\Lambda=\left(\begin{matrix}
\lambda_{11} &  0  & \ldots & 0\\
0  &  \lambda_{22} & \ldots & 0\\
\vdots & \vdots & \ddots & \vdots\\
0  &   0       &\ldots & \lambda_{d_{\xi_n}d_{\xi_n}}
\end{matrix}\right),
$$
with $\lambda_{ii}(x_{\xi_n},\xi_n)$ being the eigenvalues of 
$\sigma(x_{\xi_n},\xi_n)\sigma(x_{\xi_n},\xi_n)^{\ast}.$ 
Now let 
$$
\lambda_{mm}(x_{\xi_n},\xi_n)=\min_{1\leq i\leq d_{\xi}}\lambda_{ii}(x_{\xi_n},\xi_n).
$$
We want to show that
\begin{equation} \label{EQ:aux2}
\|u_{\xi_n}(x)U\Lambda U^{\ast}\|^2_{\HS}\geq \lambda^{2}_{mm}\|u_{\xi_n}(x)\|^{2}_{\HS}.\end{equation}
Let  $M=U\Lambda U^{\ast}.$ 
Then $M$ is symmetric and $\lambda_{mm}(x_{\xi_n},\xi_n)$ 
is the minimum eigenvalue of the matrix $M>0$.  
Let $w_{\xi_n}:=u_{\xi_n}(x) M.$
To prove \eqref{EQ:aux2}
 it is enough to show that
$$
\|w_{\xi_n}\|_{\HS}\geq \lambda_{mm}(x_{\xi_n},\xi_n)\|w_{\xi_n}M^{-1}\|_{\HS},
$$ 
where $M^{-1}=U\Lambda^{-1}U^{\ast}.$ 
This is true since $\lambda_{mm}^{-1}(x_{\xi_n},\xi_n)$ is the maximum eigenvalue of 
$M^{-1},$ that is, $\|M^{-1}\|_{op}=\lambda^{-1}_{mm}(x_{\xi_n}, \xi_n).$ 
This proves \eqref{EQ:aux2}.

\medskip
Using \eqref{EQ:aux4} and \eqref{EQ:aux2}, we can estimate 
\begin{eqnarray}
\|u\|_{L^2(G)}\|T_{\sigma}-K\|_{\L(L^{2}(G))}&\geq& 
\frac{\left(\int_{G}\|u_{\xi_n}\sigma(x_{\xi_n},\xi_n)
\sigma(x_{\xi_n},\xi_n)^{\ast}\|^2_{\HS} dx\right)^{1/2}}{\|\sigma(x_{\xi_n},\xi_n)\|_{op}}
-3 \epsilon \|u\|_{L^2(G)}\nonumber\\
&\geq& \frac{\left( \lambda^2_{mm}\int_{G}\|u_{\xi_n}(x)\|^2_{\HS} dx\right)^{1/2}}
{\|\sigma(x_{\xi_n},\xi_n)\|_{op}}-3 \epsilon \|u\|_{L^2(G)}\nonumber\\
&=&\frac{\lambda_{mm}\|u_{\xi_n}\|_{L^2(G)}}{\|\sigma(x_{\xi_n},\xi_n)\|_{op}}
-3 \epsilon \|u\|_{L^2(G)}\nonumber\\
&=&\frac{\lambda_{mm}}{\|\sigma(x_{\xi_n},\xi_n)\|_{op}}\|u\|_{L^2(G)}-3\epsilon \|u\|_{L^2(G)}\nonumber\\
&=& \left(\frac{\lambda_{mm}(x_{\xi_{n}},\xi_{n})}
{\|\sigma(x_{\xi_n},\xi_n)\|_{op}}-3\epsilon\right)\|u\|_{L^2(G)}.
\end{eqnarray}
Now, as $\jp{\xi_n}\rightarrow \infty$, we have
$$\|u\|_{L^2(G)}\|T_{\sigma}-K\|_{\ast}\geq \left(d_{\min}-3\epsilon\right)\|u\|_{L^2(G)},$$ 
that is, for any $\epsilon>0$,
$$
\|T_{\sigma}-K\|_{\ast}\geq d_{\min}-3\epsilon.
$$
Now, using the fact that $\epsilon$ is an arbitrary positive number, 
we have
$$
\|T_{\sigma}-K\|_{\ast}\geq d_{\min}.
$$
This completes the proof of Theorem \ref{THM:Gohberg}.

\begin{proof}[Proof of Lemma \ref{LEM:Tsigma-u}]
Let $x,z\in G.$ 
Let us define 
$$
R(x,z):=\sum_{[\xi]\in\Gh}d_{\xi}\Tr\left(\sigma(x,\xi)\xi(z)\right).
$$
Now we can write
\begin{eqnarray}\label{EQ:aux3}
T_{\sigma}u(x)&=&\sum_{[\xi]\in\Gh}d_{\xi}\Tr\left(\xi(x)\sigma(x,\xi)\widehat{u}(\xi)\right)\nonumber\\
&=&
\int_{G}\sum_{[\xi]\in\Gh}d_{\xi}\Tr\left(\xi(x)\sigma(x,\xi)\xi^{\ast}(y)\right)u(y)\ dy\nonumber\\
&=&
\int_{G}\sum_{[\xi]\in\Gh}d_{\xi}\Tr\left(\sigma(x,\xi)\xi(y^{-1}x)\right)u(y)\ dy\nonumber\\
&=&
\int_{G}R(x,y^{-1}x)\ u(y)\ dy\nonumber\\
&=&
\int_{G}R(x,z)\ u(xz^{-1})\ dz,
\end{eqnarray} 
where
$z=y^{-1}x.$ 
Then from the definition of $u_{\xi_n}(x)=d_{\xi_n}^{-1/2}\xi_{n}(x)u(xx^{-1}_{\xi_n})$  
and \eqref{EQ:aux3} we obtain
\begin{eqnarray}
T_{\sigma}u_{\xi_n}(x)&=&d_{\xi_n}^{-1/2}\int_{G}R(x,z)\xi_{n}(xz^{-1})u(xz^{-1}x^{-1}_{\xi_n})dz\nonumber\\
&=&d_{\xi_n}^{-1/2}\int_{G}R(x,z)\xi_{n}(x)u(xz^{-1}x^{-1}_{\xi_n})\xi_{n}^{\ast}(z)dz. 
\end{eqnarray}
Let us denote
\begin{equation}\label{EQ:vuxi}
v^{x}_{\xi_n}(z^{-1}):=v_{\xi_n}(xz^{-1}):=d_{\xi_n}^{-1/2}\xi_{n}(x)u(xz^{-1}x^{-1}_{\xi_n}),
\end{equation}
so that we have
\begin{equation}\label{EQ:Tsigma-final}
T_{\sigma}u_{\xi_n}(x)=\int_{G}R(x,z)v^{x}_{\xi_n}(z^{-1})\xi_{n}^{\ast}(z)dz.
\end{equation}
For a given collection of $m$ strongly admissible difference operators
$\Delta_1,\ldots,\Delta_m$ 
with the corresponding functions 
$q_1,\ldots,q_m\in C^{\infty}(G)$ with $\Delta_j\widehat f (\xi)=\widehat{q_j f}(\xi)$ 
we have the Taylor expansion formula, see 
\cite{ruzhansky+turunen-book, Ruzhansky-Turunen-Wirth:arxiv},
$$
u(x)=u(e)+\sum_{|\alpha|=1}^{N-1}\frac{1}{\alpha!}q^{\alpha}(x^{-1})\partial^{(\alpha)}u(e)
+{\mathcal{O}}(h(x)^N),$$ 
where $h(x)\rightarrow 0,$ $h(x)$ is the geodesic distance from $x$ and $e,$ and 
$\partial_x^{(\alpha)}$ are some left-invariant differential operators on $G$, and
$q^{\alpha}(x)=q_{1}(x)^{\alpha_1}\cdots q_{m}(x)^{\alpha_m}.$
Assuming that $u$ is sufficiently smooth, from the Taylor expansion formula we have
$$
v^{x}_{\xi_n}(z^{-1})=v^{x}_{\xi_n}(e)+\sum_{|\alpha|=1}^{N-1}\frac{1}{\alpha!}q^{\alpha}(z)\partial^{(\alpha)}v^{x}_{\xi_n}(e)+{\mathcal{O}}(h(z)^N).
$$ 
Then by the left-invariance of $\partial_x^{(\alpha)}$ we obtain
\begin{equation}\label{EQ:pf-Taylor}
v_{\xi_n}(xz^{-1})=v_{\xi_n}(x)+
\sum_{|\alpha|=1}^{N-1}\frac{1}{\alpha!}q^{\alpha}(z)\partial^{(\alpha)}v_{\xi_n}(x)+
{\mathcal{O}}(h(z)^N).
\end{equation}
Using \eqref{EQ:Tsigma-final} and \eqref{EQ:pf-Taylor}, we can now write
\begin{eqnarray}
T_{\sigma}u_{\xi_n}(x)= 
\int_{G}R(x,z)v_{\xi_n}(x)\xi_{n}^{\ast}(z)dz
&+&
\int_{G}R(x,z)\sum_{|\alpha|=1}^{N-1}\frac{1}{\alpha!}
q^{\alpha}(z)\partial^{(\alpha)}v_{\xi_n}(x)\xi_{n}^{\ast}(z)dz\nonumber\\
&+&
\int_{G}R(x,z){\mathcal{O}}(h(z)^N)\xi_{n}^{\ast}(z)dz. \nonumber
\end{eqnarray}
We denote 
$$
I_{1}:=\int_{G}R(x,z)v_{\xi_n}(x)\xi_{n}^{\ast}(z)dz,
$$ 
$$
I_2:=\int_{G}R(x,z)\sum_{|\alpha|=1}^{N-1}\frac{1}{\alpha!}
q^{\alpha}(z)\partial^{(\alpha)}v_{\xi_n}(x)\xi_{n}^{\ast}(z)dz,
$$ 
and 
$$
I_3:=\int_{G}R(x,z){\mathcal{O}}(h(z)^N)\xi_{n}^{\ast}(z)dz.
$$
Then we have
\begin{eqnarray}
I_1&=&\int_{G}R(x,z)v_{\xi_n}(x)\xi_{n}^{\ast}(z)dz\nonumber\\
&=& v_{\xi_n}(x)\sigma(x,\xi_n)\nonumber\\
&=&d_{\xi_n}^{-1/2}\xi_n(x)u(xx^{-1}_{\xi_n})\sigma(x,\xi_n)
=u_{\xi_n}(x)\sigma(x,\xi_n).
\end{eqnarray}
Calculating $I_2$, we have 
\begin{eqnarray}
I_{2}&=&\int_{G}R(x,z)  \partial^{(\alpha)}v_{\xi_n}(x)
\xi_{n}^{\ast}(z)\sum_{|\alpha|=1}^{N-1}\frac{1}{\alpha!}
q^{\alpha}(z)dz\nonumber\\
&=&\sum_{|\alpha|=1}^{N-1}\frac{1}{\alpha!}\partial^{(\alpha)}v_{\xi_n}(x)\int_{G}R(x,z)\xi_{n}^{\ast}(z)q^{\alpha}(z)dz\nonumber\\
&=&\sum_{|\alpha|=1}^{N-1}\frac{1}{\alpha!}\partial^{(\alpha)}v_{\xi_n}(x)\left({\mathcal{F}}q^{\alpha}{\mathcal{F}}^{-1}\sigma\right)(\xi_n)\nonumber\\
&=& \sum_{|\alpha|=1}^{N-1}\frac{1}{\alpha!}
\partial^{(\alpha)}v_{\xi_n}(x)
\Delta_{q^{\alpha}}\sigma(x,\xi_n).
\end{eqnarray}
And calculating $I_3$, we have
\begin{eqnarray}
I_3&=&\int_{G}R(x,z) {\mathcal{O}}(h(z)^N) \xi_{n}^{\ast}(z) dz\nonumber\\
&=&\int_{G}R(x,z)q^{N}(z) \xi_{n}^{\ast}(z)dz\nonumber\\
&=&\Delta_{q^{N}}\sigma(x,\xi_n),
\end{eqnarray}
where we can denote 
$q^N:=\mathcal{O}(h(x)^N)$ so that $q^{N}$ vanishes at $e$ of order $N$,
but keep in mind that it is matrix-valued.
Then we have
\begin{eqnarray}
T_{\sigma}u_{\xi_n}(x)-u_{\xi_n}(x)\sigma(x,\xi_n)
&=& 
\sum_{|\alpha|=1}^{N-1}\frac{1}{\alpha!}
\partial^{(\alpha)}v_{\xi_n}(x)\Delta_{q^{\alpha}}\sigma(x,\xi_n)\nonumber\\
&+& 
\Delta_{q^{(N)}}\sigma(x,\xi_n).
\end{eqnarray}
We denote 
$$
T^{1}_N:=\sum_{|\alpha|=1}^{N-1}\frac{1}{\alpha!}
\partial^{(\alpha)}v_{\xi_n}(x) \Delta_{q^{\alpha}}\sigma(x,\xi_n)
$$ 
and
$$
T^{2}_{N}:= 
\Delta_{q^{N}}\sigma(x,\xi_n). 
$$
We can estimate
$$
\|T^{1}_N(x)\|_{\HS}\leq 
\|\Delta_{q^{\alpha}}\sigma(x,\xi_n)\|_{op}
\sum_{0<|\alpha|\leq N}\frac{1}{\alpha!}\|\partial^{(\alpha)}_zv^{x}_{\xi_n}(x)\|_{\HS}
$$ 
where
$z\in G.$
So, using \eqref{EQ:vuxi} and \eqref{EQ:symbol-class}, 
for some operator $\tilde{\partial}^{(\alpha)}_{z}$ we have

\begin{eqnarray}
\|T^{1}_N(x)\|_{\HS}
&\leq& 
C \left(\sum_{0<|\alpha|\leq N}|\tilde{\partial}^{(\alpha)}_{z}
u(x\cdot x_{\xi^{-1}_n})|\right) d^{-1/2}_{\xi_n}\|\xi_n(x)\|_{\HS}
\langle\xi_n\rangle^{-|\alpha|}\nonumber\\
&\leq& C \sum_{0<|\alpha|\leq N}\|u\|_{H^{|\alpha|}}\langle\xi_n\rangle^{-|\alpha|}\nonumber\\
&\leq& C^{\prime} \langle\xi_n\rangle^{-1},
\end{eqnarray} 
where
$1\leq |\alpha|\leq N,$ $N$ is fixed and $\|\cdot\|_{H^{|\alpha|}}$ 
is the Sobolev norm.
Similarly, 
$$
\|T^{2}_{N}\|_{\HS}\leq C\langle\xi_n\rangle^{-N}.
$$
Therefore,
as $\langle\xi_n\rangle\rightarrow \infty$ we have 
$\|T^{1}_{N}(x)\|_{\HS}\rightarrow 0$ and 
$\|T^{2}_N(x)\|_{\HS}\rightarrow 0$
for all $x\in G$ which gives 
$$
\|T^{1}_{N}\|^{2}_{L^{2}(G)}=\int_{G}\|T^{1}_N(x)\|^{2}_{\HS}dx\rightarrow 0
$$ 
as 
$\jp{\xi_n}\rightarrow \infty$ and, similarly, 
$\|T^{2}_N\|_{L^2(G)}\rightarrow 0.$ 
This implies
$$
 \|T_{\sigma}u_{\xi_n}-u_{\xi_n}\sigma(\cdot,\xi_n)\|_{L^{2}(G)}\rightarrow 0
$$ 
as $\jp{\xi_n}\rightarrow \infty,$ and we note that it is sufficient to take
$N=1$ in the above argument. 
\end{proof} 

\begin{rem} \label{REM:Gohberg}
Looking at what we have used in the proof, we note that we have
(with the same proof and $N=1$) the following extension of the
Gohberg lemma without making an assumption that
the operator belongs to $\Psi^{0}(G)$, namely:

\medskip
{\em Let $T_{\sigma}:L^{2}(G)\to L^{2}(G)$ be a bounded operator with
the matrix symbol $\sigma(x,\xi)$ satisfying, for some $\rho>0$,
$$
\|\sigma(x,\xi)\|_{op}\leq C,\quad
\|\partial_{x}\sigma(x,\xi)\|_{op}\leq C,
\quad
\|\Delta_{q}\sigma(x,\xi)\|_{op}\leq C\jp{\xi}^{-\rho}
$$
for all $q\in C^{\infty}(G)$ with $q(e)=0$, and all
$x\in G$ and $[\xi]\in\Gh$.  Then the conclusion of 
the Gohberg lemma in Theorem \ref{THM:Gohberg} remains true,
namely, the estimate \eqref{EQ:Gohberg-est} holds for all
compact operators $K$ on $L^2(G)$.}
\end{rem}

\section{Proof of Theorem \ref{THM:Thmpm}}
\label{SEC:other-proofs}

We first recall the following theorem which is known as the Atkinson theorem 
which gives another equivalent definition of Fredholm operators.

\begin{thm}\label{THM:Atkinson}
Let $A$ be a closed linear operator from a complex Banach space $X$ into a
complex Banach space $Y$ with a dense domain $D(A)$. Then $A$ is Fredholm if and only if
we can find a bounded linear operator $B : Y \rightarrow X$, a compact operator 
$K_1 : X \rightarrow X$ and
a compact operator $K_2 : Y \rightarrow Y$ such that 
$BA = I + K_1$ on $D(A)$ and $AB = I + K_2$ on
$Y$.
\end{thm}

We recall that the Wolf spectrum $\Sigma_{w}(A)$ of $A$ is defined by 
$\Sigma_{w}(A):=\mathbb{C}\backslash\Phi_{w}(A),$ 
where 
$$
\Phi_{w}(A)=\left\{\lambda\in\mathbb{C}: A-\lambda I \textrm{ is Fredholm}\right\}.
$$
Clearly, we have $\Sigma_{w}(A)\subseteq \Sigma_{ess}(A)\subseteq \Sigma(A)$.

\begin{proof}[Proof of Theorem \ref{PROP:1}]
Let $\lambda\in \mathbb{C}$ be such that $|\lambda| > d_{\max}$.
Then there exists $\epsilon>0$ such that 
$$|\lambda|>d_{\max}+\epsilon.$$
Now, by the definition of $d_{\max}$ in Theorem
\ref{THM:Thmpm}, we have for some $R>0$ and
for all $\jp{\xi}\geq R$ that  
$$\sup_{\jp{\xi}\geq R}\{\sup_{x\in G}\|\sigma(x,\xi)\|_{op}\}\leq ( d_{\max}+\epsilon/2).$$
Then for $\jp{\xi}\geq R$, we can estimate
\begin{multline}
\|(\sigma(x,\xi)-\lambda I)^{-1}\|_{op}
\leq \sum_{k=1}^{\infty}\lambda^{-k-1}\|\sigma(x,\xi)^{k}\|_{op} \\
\leq \sum_{k=1}^{\infty} \frac{(d_{\max}+\epsilon/2)^{k}}{|\lambda|^{k+1}} 
 \leq \sum_{k=1}^{\infty} \frac{(d_{\max}+\epsilon/2)^{k}}{{(d_{\max}+\epsilon)}^{k+1}}
 \leq C \sum_{k=1}^{\infty} \frac {(d_{\max}+\epsilon/2)^{k}}{{(d_{\max}+\epsilon)}^{k}} 
 < \infty.
\end{multline}
Hence from 
\eqref{EQ:ellipticity} it follows that 
the operator $T_{\sigma}-\lambda I$ is elliptic and hence it is a Fredholm operator from 
$L^{2}(G)\rightarrow L^{2}(G)$, see e.g. 
\cite[Section 19.5]{Hormander:BK-Vol-III}.
Thus 
$$
\left\{ \lambda\in {\mathbb{C}}: |\lambda|>d_{\max} \right\}\subseteq \Phi_{w}(T_{\sigma}),
$$ 
which implies that 
$$
\Sigma_{w}(T_{\sigma})\subseteq\left\{\lambda\in \mathbb{C}: |\lambda|\leq d_{\max}\right\}.
$$ 
Since $\{ \lambda\in {\mathbb{C}}: |\lambda|> d_{\max}\}$ is a connected component of 
$\Phi_{w}(T_{\sigma})$ it follows that $i(T_{\sigma}-\lambda I)$ is constant for all 
$\lambda$ in $\{\lambda\in {\mathbb{C}}: |\lambda|> d_{\max}\}.$ 
Now again 
$$
\Phi(T_{\sigma})\cap \{\lambda\in {\mathbb{C}}: |\lambda|> d_{\max}\}\neq \emptyset.$$ 
Therefore $i(T_{\sigma}-\lambda I)=0$ for all $\{\lambda\in \mathbb{C}: |\lambda|> d_{\max}\}.$ 
This implies 
$$
\Sigma_{ess}(T_{\sigma})\subseteq \{\lambda\in \mathbb{C}: |\lambda|\leq d_{\max}\},
$$
completing the proof of \eqref{EQ:ess}.

\medskip
To prove the last part of Theorem \ref{THM:Thmpm}, we start by
recalling the definition of the Calkin algebra. Let $\L(L^2(G))$ and 
$\K(L^2(G))$ be respectively the $C^{\ast}$ algebra of bounded linear operators on 
$L^2(G)$ and  the ideal of compact operators on $L^{2}(G).$ 
The Calkin algebra, $\L(L^{2}(G))/\K(L^{2}(G))$, is a $\ast$-algebra. 
The product and the involution are defined here as
$$[A][B]:=[AB]$$ 
and 
$$[A]^{\ast}:=[A^{\ast}],$$ 
for all $A$ and $B$ in $\L(L^{2}(G)).$  
Let $[A]$ and $[B]$ be members of $\L(L^{2}(G))/\K(L^{2}(G))$. 
Then 
$$[A]=[B] \Longleftrightarrow  A-B\in \K(L^{2}(G)).$$ 
The norm $\|\cdot\|_C$ in $\L\left(L^2(G))/\K(L^2(G)\right)$ is given by 
$$\|[A]\|_C:=\inf_{K\in \K(L^2(G))}\|A-K\|_{\L(L^{2}(G))}, \quad[A] \in \L\left(L^2(G))/\K(L^2(G)\right).$$
By using the Calkin algebra the Gohberg lemma Theorem
\ref{THM:Gohberg} can be reformulated as the inequality
$$
\|[T_{\sigma}]\|_{C}\geq d_{\min}.
$$

\medskip
We now prove the last part of Theorem \ref{THM:Thmpm}.
We assume that $d_{\max}=0$ and observe that  $T_{\sigma}$ is compact if and only if 
$[T_{\sigma}]=0$ in the Calkin algebra $\L(L^2(G))/\K(L^2(G)).$   
We also observe that the operator $T_{\sigma}$ is essentially normal on 
$L^{2}(G)$, i.e. $T_{\sigma}T_{\sigma}^{*}-T_{\sigma}^{*}T_{\sigma}$ is compact.
Indeed, this is the operator of order $-1$ so the compactness follows from
the compactness of the embedding $H^{1}\hookrightarrow L^{2}.$
Consequently, $[T_{\sigma}]$ is normal in  $\L(L^2(G))/\K(L^2(G)),$ 
and, therefore,
$$
r(T_{\sigma})=\|[T_{\sigma}]\|_{C},
$$ 
where $r(T_{\sigma})$ is the spectral radius of $[T_{\sigma}].$ 
On the other hand we know that $\Sigma_{ess}(T_{\sigma})\subset\{0\}$ by 
the first part of Theorem \ref{PROP:1} in \eqref{EQ:ess}. This implies that 
$T_{\sigma}-\lambda I$ is Fredholm for $\lambda\neq 0.$ 
So using Atkinson's Theorem \ref{THM:Atkinson} 
this implies that there exists a bounded operator $B$ such that 
$$(T_{\sigma}-\lambda I)B=I+K,$$ 
where $K$ is a compact operator. 
That is, for $\lambda\neq 0$, $[(T_{\sigma}-\lambda I)]$ is invertible, 
which implies that for $\lambda\neq 0$, we have $\lambda\notin \Sigma([T_{\sigma}]),$ 
the spectrum of $[T_{\sigma}]$. So $\Sigma([T_{\sigma}])\subseteq \{0\}.$ 
Consequently,
$$\|[T_{\sigma}]\|_{C}=r(T_{\sigma})=0.$$ 
Therefore $[T_{\sigma}]=0,$ 
and hence $T_{\sigma}$ is compact, completing the proof.
\end{proof}



\end{document}